\documentclass[a4paper,10pt,abstract=on]{scrartcl}

\usepackage{sidecap}
\usepackage[utf8]{inputenc}
\usepackage[T1]{fontenc}
\usepackage[english]{babel}
\usepackage{hyphenat}
\usepackage{amsmath,amssymb,amsfonts}
\usepackage{amsthm}
\usepackage[svgnames]{xcolor}
\usepackage{mathtools}
\mathtoolsset{showonlyrefs}
\usepackage{geometry}
\usepackage{bm}
\usepackage[shortlabels]{enumitem}
\usepackage{csquotes}
\usepackage{booktabs}
\usepackage{cite}
\usepackage{subcaption}
\usepackage[hidelinks]{hyperref}
\usepackage{multirow}
\usepackage{multicol}

\usepackage[svgnames]{xcolor}
\usepackage{mathtools}
\usepackage{bm}
\usepackage{csquotes}
\usepackage{booktabs}
\mathtoolsset{showonlyrefs}
\usepackage[shortlabels]{enumitem}
\usepackage[hidelinks]{hyperref}
\usepackage{comment}
\usepackage{todonotes}

\theoremstyle{plain}
\newtheorem{theorem}{Theorem} 
\newtheorem{lemma}[theorem]{Lemma}

\theoremstyle{definition}

\newtheorem{algorithm}[theorem]{Algorithm}

\theoremstyle{remark}
\newtheorem{remark}[theorem]{Remark}

\providecommand{\keywords}[1]{
  \small	
  \textbf{\textit{Keywords---}} #1}

\providecommand{\msc}[1]{
  \small	
  \textbf{\textit{2020 AMS Mathematics Subject Classification---}} #1}

\DeclareMathOperator*{\argmin}{arg\,min}

\DeclareMathOperator{\prox}{prox}
\DeclareMathOperator{\conv}{conv}
\DeclareMathOperator{\TV}{TV}

\DeclareMathOperator{\proj}{proj}

\newcommand{\N}{\mathbb N}

\newcommand{\R}{\mathbb R}

\newcommand{\Fun}[1]{\mathcal{#1}}
\newcommand{\FF}{\Fun F}

\newcommand{\FK}{\Fun K}

\newcommand{\FQ}{\Fun Q}

\newcommand{\FL}{\Fun L}

\newcommand{\Mat}[1]{\bm{#1}}
\newcommand{\MA}{\Mat A}

\newcommand{\ME}{\Mat E}
\newcommand{\MI}{\Mat I}

\newcommand{\ML}{\Mat L}

\newcommand{\MQ}{\Mat Q}

\newcommand{\MU}{\Mat U}
\newcommand{\MV}{\Mat V}

\newcommand{\MX}{\Mat X}
\newcommand{\MY}{\Mat Y}
\newcommand{\MZ}{\Mat Z}
\newcommand{\Mzero}{\Mat 0}
\newcommand{\Mone}{\Mat 1}

\newcommand{\Vek}[1]{\bm{#1}}

\newcommand{\Vx}{\Vek x}
\newcommand{\Vy}{\Vek y}
\newcommand{\Vz}{\Vek z}

\newcommand{\Vu}{\Vek u}

\newcommand{\Vxi}{\Vek \xi}

\newcommand{\Veta}{\Vek \eta}

\newcommand{\sphere}{\mathbb S}

\newcommand{\Binary}{\mathbb B}
\newcommand{\Ball}{\mathbb C}

\newcommand{\stiefel}{\mathbb V}

\newcommand{\SO}{{\mathrm{SO}}}
\newcommand{\PSD}{{\mathbb{P}}}

\newcommand{\tF}{\mathrm{F}}

\newcommand{\rk}{\mathrm{rk}}

\begin{document}

\title{Denoising Multi-Color QR Codes \\
and Stiefel-Valued Data \\
by Relaxed Regularizations}

\author{Robert Beinert\thanks{R. Beinert is with the Institute of Mathematics,
	Technische Universit\"at Berlin, Stra\ss{}e des 17. Juni 136,
        10623 Berlin, Germany.}
        \qquad
        Jonas Bresch\thanks{J. Bresch is with the Institute of Mathematics,
	Technische Universit\"at Berlin, Stra\ss{}e des 17. Juni 136,
        10623 Berlin, Germany.}
}

\maketitle

\begin{abstract}
    The handling of manifold-valued data,
    for instance,
    plays a central role in color restoration tasks
    relying on circle- or sphere-valued color models,
    in the study of rotational or directional information
    related to the special orthogonal group,
    and in Gaussian image processing,
    where the pixel statistics are interpreted as
    values on the hyperbolic sheet.
    Especially,
    to denoise these kind of data,
    there have been proposed several 
    generalizations of total variation (TV)
    and Tikhonov-type denoising models
    incorporating the underlying manifolds.
    Recently,
    a novel, numerically efficient denoising approach
    has been introduced,
    where the data are embedded in an Euclidean ambient space,
    the non-convex manifolds are encoded by a series
    of positive semi-definite, fixed-rank matrices,
    and the rank constraint is relaxed to obtain
    a convexification 
    that can be solved using standard algorithms
    from convex analysis.
    The aim of the present paper is to extent this approach
    to new kinds of data
    like multi-binary and Stiefel-valued data.
    Multi-binary data can, for instance, be used to model multi-color QR codes
    whereas Stiefel-valued data occur in image and video-based recognition.
    For both new data types,
    we propose TV- and Tikhonov-based denoising models
    together with easy-to-solve convexification. 
    All derived methods are evaluated on
    proof-of-concept, synthetic experiments.

\end{abstract}


\hspace{0.25cm}

\msc{94A08, 94A12, 65J22, 90C22, 90C25}

\section{Introduction}
\label{sec:intro}

Due to novel emerging acquisition techniques,
data processing on manifold-valued signals and images 
becomes more and more important. 
Especially,
when the data manifold is non-convex,
classical data processing tasks like denoising
turn into challenging novel problems.
In real-world applications,
data with values on non-convex manifolds,
for instance,
play a major role    
in color restoration tasks relying on the HSV (hue-saturation-value)
or LCh (lightness-chromaticity-hue) color space \cite{NikSte14,LEHS2008},
where the hue is represented by a circle,
and the color component is thus interpreted as circle-valued data.
For denoising tasks relying on the
chromaticity-brightness model \cite{PPS2017,QKL2010},
the pixels of an image instead attain values on the sphere.
Rotation and directional information are modeled using the special orthogonal group;
hence the corresponding data are often $\SO(2)$- or $\SO(3)$-valued \cite{ASWK1993,BHJPSW10}.
For Gaussian image processing, 
mean and variance of the pixel statistic can be interpreted as points on the hyperbolic sheet
yielding hyperbolic-valued data \cite{BerPerSte16,L97}.
To deal with data of these kinds,
especially in the context of denoising,
the classic variational Tikhonov and total variation (TV) models
have been generalized 
relying on non-convex optimization methods on the underlying manifolds
\cite{LSKC13,CS13,LauNikPerSte17,BerLauSteWei14,Ken23,BerChaHiePerSte16,GS14}.
All these methods come with some limitations,
they either increase the dimension using lifting techniques,
require a tree structure of the graph where the data is living on,
or are based on a subclass of the considered manifolds.
A completely novel approach is proposed
in \cite{BeBr24,BeBr25,BeBrSt25},
where the manifold-valued data are embedded into an Euclidean ambient space
allowing the use of standard denoising models.
The non-convex side constrains caused by the non-convex manifolds
are moreover encoded by a series of positive semi-definite, fixed-rank matrices.
Relaxing the fixed rank results in a convexification of the denoiser model
that may be solved using standard algorithms from convex analysis.

The aim of this paper is to extend the proposed encoding and convexification techniques
in \cite{BeBr24,BeBr25,BeBrSt25}
to new types of data,
namely multi-binary
and Stiefel-valued data.
Multi-binary signals with values in $\Binary_d = \{-1,1\}^d \subset \R^d$, 
occur in the context of multi-color QR codes
for instance. 
The main idea behind multi-color QR codes is
to increase the information density and information capacity
by combining three independent black/white QR codes
with respect to the color channels in the RGB (red-blue-green) color model
\cite{Ta18,Mi16,An24}.
The modules of the resulting QR codes are colored in black, red, blue, green, cyan, magenta, yellow, and white.
Since these colors correspond to the vertices of the RGB color space,
rescaling and shifting the color space allow us
to interpret multi-color QR codes as $\Binary_3$-valued images.
Notice that $\Binary_d$ is actually no manifold.
Nevertheless,
the TV methods proposed in \cite{BeBr24} can be generalized 
to handle this kind of data.
The Stiefel manifold $\stiefel_d(k)$ consists of all $k$-tuples
of orthonormal vectors in $\R^d$;
thus the Stiefel manifold figuratively consists
of all orthogonal bases of all $k$-dimensional subspaces of $\R^d$,
where the order of the basis vectors matter.
The Stiefel manifold arises in 
image and video-based recognition \cite{TVSC2011}.
The handling orthonormal constraints is, for instance, studied \cite{WY2013,EAS1998},
where geodesic- and lifting processes play a major part. 
In the following,
we consider general signals supported 
on a connected, undirected graph $G = (V, E)$,
where $V \coloneqq \{1, \dots, N\}$ is the set of vertices
and $E \coloneqq \{(n,m) : n < m\} \subset V \times V$ the set of edges,
encoding the structure of the data.
We set $M \coloneqq |E|$ to be the number of edges.

This paper is organized as follows:
In §~\ref{sec:tv_binary_valued_problem}
we derive a convex relaxation for a $\Binary_d$-valued TV denoising model
and a numerical solver based on the Alternating Direction Method of Multipliers (ADMM).
In Theorem~\ref{thm:binary_prob_tightness_tv},
we show that our convexification is actually tight,
meaning that
every solution of the convexified model can be used 
to easily construct a solution of the original non-convex model.
The denoising of $\stiefel_d(k)$-valued data is studied in §~\ref{sec:stiefel},
where we propose a TV model for denoising cartoon-like or piecewise constant data
and a Tikhonov model for denoising smooth data.
For both variational denoising models,
we again derive numerical solvers relying on the ADMM.
In §~\ref{sec:num},
we show the success for our proposed methods 
by applying the resulting algorithms 
to denoise multi-color QR codes
and Stiefel-valued signals.

\section{TV Denoising for Multi-Binary Data}
\label{sec:tv_binary_valued_problem}

The aim of this section is to generalize the TV-denoising model
for binary data in \cite{BeBr24} to multi-binary data,
which, 
for instance, 
occurs in the context of multi-color QR codes.
More precisely,
multi-binary data takes on values in $\Binary_d \coloneqq \{-1,1\}^d \subset \R^d$.
For a graph $G \coloneqq (V,E)$ as introduced in §~\ref{sec:intro},
we are interested in 
restoring a multi-binary signal $\Vx \coloneqq (\Vx_n)_{n \in V} \subset \Binary_d$
from noisy measurement $\Vy \coloneqq (\Vy_n)_{n \in V} \subset \R^d$.
Taking the piecewise constant or cartoon-like structure of multi-color QR codes into account,
we introduce the following TV denoiser:
\begin{align}
    \label{eq:tv_binary_non_convex}
    \argmin_{\Vx \in \Binary_d^N} 
    \; 
    \tfrac{1}{2} \sum_{n \in V} \|\Vx_n - \Vy_n\|^2_2 
    + \lambda \TV(\Vx),
    \quad 
    \text{where}
    \quad
    \TV(\Vx) \coloneqq \sum_{(n,m) \in E} \|\Vx_n - \Vx_m\|_1,
\end{align}
which relies on an anisotropic TV regularization on the ambient space $\R^d$,
and where $\lambda$ is a positive regularization parameter.
Since the domain $\Binary_d^N$ is non-convex,
we thus have to solve a non-convex minimization problem.
To convexify \eqref{eq:tv_binary_non_convex},
we first exploit 
that $\lVert \Vx_n\rVert^2 = 1$
and that $\lVert \Vy_n \rVert^2$ is fixed by the given data;
therefore,
we have 
$\tfrac{1}{2}\|\Vx_n - \Vy_n\|_2^2
= -\langle \Vx_n, \Vy_n\rangle + \text{constant}$
such that the objective in \eqref{eq:tv_binary_non_convex}
can be replaced by
\begin{align*}
    \FK(\Vx)
    \coloneqq 
    -\sum_{n \in V} \; \langle \Vx_n, \Vy_n\rangle + \lambda \TV(\Vx).
\end{align*}
Second,
we convexify the domain of \eqref{eq:tv_binary_non_convex} 
via replacing $\Binary_d$ by its convex hull  
$\Ball_d \coloneqq \conv(\Binary_d) = [-1,1]^d$,
which in fact yields a cube in $\R^d$.
Instead of solving \eqref{eq:tv_binary_non_convex},
we propose to solve the \textbf{convexified TV model}:
\begin{align}
    \label{eq:tv_binary_convex}
    \argmin_{\Vx \in \R^{d\times N}} 
    \quad
    \FK(\Vx)
    \quad 
    \text{s.t.}
    \quad 
    \Vx_n \in \Ball_d
    \quad 
    \forall
    n \in V.
\end{align}
Heuristically,
the term $-\langle \Vx_n, \Vy_n \rangle$ in the objective $\FK$
pushes $\Vx_n$ in direction $\Vy_n$
until it hits the boundary
\begin{equation*}
    \partial \Ball_d = \{\Vxi \in \Ball_d : \exists i \in \{1,...,d\} \;\text{s.t.}\; |\Vxi_i| = 1 \}
\end{equation*}
and, in the best case, a corner in $\Binary_d$.

\paragraph{Tightness of the Convexified TV Model}

In analogy to binary signals \cite{BeBr24},
the convexification \eqref{eq:tv_binary_convex} is in fact tight.
This means that
every solution of \eqref{eq:tv_binary_convex} 
can be used to obtain a solution of the original non-convex problem 
\eqref{eq:tv_binary_non_convex}.
For the actual transference of $\Vx_n = (\Vx_{n,1}, \dots, \Vx_{n,d})^* \in \R^d$,
we define $X_{\Veta} \colon \R^{d} \to \Binary_d^N$ 
with $\Veta = (\Veta_1, \dots, \Veta_d)^* \in \Ball_d$ via
\begin{align*}
    X_{\Veta}(\Vx_n)
    \coloneqq 
    (\chi_{\Veta_i}(\Vx_{n,i}))_{i = 1}^d
    \quad\text{with}\quad
    \chi_{\Veta_i}(\Vx_{n,i})
    \coloneqq 
    \begin{cases}
    1 & \text{if } \Vx_{n,i} > \Veta_i, \\
    -1 & \text{if } \Vx_{n,i} \le \Veta_i. \\
    \end{cases}
\end{align*}
Moreover,
we introduce $X_{\Veta}(\Vx) \coloneqq (X_{\Veta}(\Vx_n))_{n \in V}$
to transfer the entire signal.
For the characteristic $X$,
we obtain the following generalized version of \cite[Lem.~3.1]{BeBr24},
which relates to the so-called coarea formula \cite{Fed59,FleRis60}.

\begin{lemma}
\label{lem:coarea}
    For $\Vx_n, \Vx_m \in \Ball_d$,
    it holds
    \begin{equation*}
        \|\Vx_n - \Vx_m\|_1 
        = \frac{1}{2^d}\int_{\Ball_d} \|X_{\Veta}(\Vx_n) - X_{\Veta}(\Vx_m)\|_1 \;\mathrm{d} \Veta.
    \end{equation*}
\end{lemma}
\begin{proof}
    For simplicity,
    we only consider the case
    $\Vx_{m,i} < \Vx_{n,i}$
    for all $i \in \{1,...,d\}$
    and obtain by Fubini's theorem
    \begin{equation*}
        \int_{\Ball_d} \|X_{\Veta}(\Vx_n) - X_{\Veta}(\Vx_m)\|_1 \;\mathrm{d} \Veta
        =
        2^{d-1}
        \sum_{i = 1}^d
        \int_{\Vx_{m,i}}^{\Vx_{n,i}} \lvert 1 - (-1)\rvert \;\mathrm{d}\Veta_i
        = 2^d \sum_{i = 1}^d \lvert \Vx_{m,i} - \Vx_{n,i}\rvert
        = 2^d\lVert\Vx_n - \Vx_m\rVert_1.
    \end{equation*}
    The general case follows analogous.
\end{proof}

\begin{theorem}
\label{thm:binary_prob_tightness_tv}
    Let $\Vx^* \in \Ball_d^{N}$
    be a solution of \eqref{eq:tv_binary_convex}.
    Then $X_{\Veta}(\Vx^*) \in \Binary_d^N$ is a solution of \eqref{eq:tv_binary_non_convex}
    for almost all $\Veta \in \Ball_d$.
\end{theorem}
\begin{proof}
    The proof follows ideas from \cite{CHOGENOB2011,BeBr24}. 
    For any $\Vx_n \in \Ball_d^{N}$
    and $\Vy_n \in \R^d$,
    we have the integral representation:
    \begin{align*}
        \frac{1}{2^d}
        \int_{\Ball_d^N}
        \langle X_{\Veta}(\Vx_n), \Vy_n \rangle
        \;\mathrm{d}\Veta 
        &= 
        \frac{1}{2}
        \sum_{i = 1}^d 
        \int_{-1}^1
        X_{\Veta_i}(\Vx_{n,i}) \, \Vy_{n,i}
        \;\mathrm{d}\Veta_i 
        = 
        \frac{1}{2}
        \sum_{i = 1}^d 
        \biggl[
        \int_{-1}^{\Vx_{n,i}}
        \Vy_{n,i} \;\mathrm d \Veta_i
        - 
        \int_{\Vx_{n,i}}^{1}
        \Vy_{n,i} \;\mathrm d \Veta_i\biggr] \\
        &=
        \frac{1}{2}
        \sum_{i = 1}^d 
        [\Vx_{n,i} - (-1) - (1 - \Vx_{n,i})] \; \Vy_{n,i}
        = 
        \langle \Vx_n, \Vy_n \rangle.
    \end{align*}
    Together with Lem.~\ref{lem:coarea},
    this yields
    \begin{align*}
        \FK(\Vx^*)
        &=
        -\smashoperator{\sum_{n \in V}} \langle \Vx_n^*, \Vy_n\rangle 
        + \lambda \smashoperator{\sum_{(n,m) \in E}} \|\Vx_n^* - \Vx_m^*\|_1   
        \\
        &= \frac{1}{2^d}
        \int_{\Ball_d} \bigg[ -\smashoperator{\sum_{n \in V}}\langle X_{\Veta}(\Vx_n^*),\Vy_n\rangle
        + \lambda \smashoperator{\sum_{(n,m) \in E}} \|X_{\Veta}(\Vx_n^*) - X_{\Veta}(\Vx_m^*)\|_1 
        \biggr]\; \mathrm{d} \Veta 
        = \frac{1}{2^d}\int_{\Ball_d} \FK(X_{\Veta}(\Vx^*))\;\mathrm{d} \Veta.
    \end{align*}
    Since $\FK(\Vx^{**}) \ge \FK(\Vx^*)$ 
    for any solution $\Vx^{**} \in \Binary_d^N$ of \eqref{eq:tv_binary_non_convex}, 
    $X_{\Veta}(\Vx^*)$ minimizes \eqref{eq:tv_binary_non_convex}
    for almost all $\Veta \in \Ball_d$ as well.
\end{proof}

\paragraph{Solving the Convexified TV Model}

The convex problem \eqref{eq:tv_binary_convex}
may be solved employing standard methods from convex analysis.
Similarly to \cite[§~3]{BeBr24},
we rely on the so-called Alternating Direction Method of Multipliers (ADMM) \cite{bauschke2017}.
For this,
we consider the splitting
$\FK(\Vx) + \iota_{\Ball_d}(\Vu)$ with $\Vx - \Vu = 0$.
Here,
$\iota_{\Ball_d} \equiv 0$ on $\Ball_d$ 
and $\iota_{\Ball_d} \equiv + \infty$ elsewhere.
The proximal mapping of a general functional $\FF\colon \R^{d \times N} \to \R$
and the projection onto a closed set $A \subset \R^{d\times N}$ are defined as
\begin{equation}
    \label{eq:prox-proj}
    \prox_{\FF, \gamma}(\Vz) 
    \coloneqq 
    \argmin_{\Vx \in \R^{d \times N}} \Bigl\{\FF(\Vx) + \tfrac{1}{2\gamma} \sum_{n \in V} \lVert \Vx_n - \Vz_n\rVert^2 \Bigr\}
    \quad\text{and}\quad
    \proj_A(\Vz) 
    \coloneqq
    \argmin_{\Vx \in A} \sum_{n \in V} \lVert \Vx_n - \Vz_n \rVert^2.
\end{equation}
Exploiting the computation rule for the proximal mapping,
we obtain Algorithm~\ref{alg:1}.

\begin{algorithm}{ADMM to solve \eqref{eq:tv_binary_convex}.}
    \label{alg:1}
    Choose $\Vx^{(0)} = \Vu^{(0)} = \Vz^{(0)}= \Mzero \in \R^N$, 
    step size $\rho > 0$,
    and regularization parameter $\lambda > 0$
    \\
    \textbf{For} $i \in \N$ \textbf{do:}\\
    \hspace*{0.5cm} 
    $\Vx^{(i+1)}
        = \prox_{\FK, \frac{1}{\rho}}(\Vu^{(i)} - \Vz^{(i)})
        = \prox_{\TV, \frac{\lambda}{\rho}}(\Vu^{(i)} - \Vz^{(i)} + \frac{1}{\rho} \,\Vy)$\\
    \hspace*{0.5cm} 
    $\Vu^{(i+1)} 
    = \prox_{\Ball_d^N, \frac{1}{\rho}}(\Vx_n^{(i+1)} + \Vz_n^{(i)})
    = \proj_{\Ball_d^N}(\Vx_n^{(i+1)} + \Vz_n^{(i)})$ \\
    \hspace*{0.5cm} $\Vz^{(i+1)} 
    = \Vz^{(i)} + \Vx^{(i+1)} - \Vu^{(i+1)}$
\end{algorithm}

\begin{remark}
    The numerical projection to $\Ball_d^N$ is unproblematic.
    Since we rely on an anisotropic TV regularization,
    the TV proximal mapping can be efficiently computed by
    applying the fast TV program \cite{Con13v4} coordinatewise.
    The convergence of Algorithm~\ref{alg:1} is guaranteed 
    by \cite[Cor.~28.3]{bauschke2017}.
\end{remark}

\section{Denoising Stiefel-Valued Data}
\label{sec:stiefel}

The (real) Stiefel-manifold $\stiefel_d(k)$ is
the union of all $k$-tuples of orthonormal vectors in $\R^d$
with respect to the Euclidean inner product.
In other words, 
$\stiefel_d(k)$ consists of all orthonormal bases
of all $k$-dimensional subspaces,
where the order of a basis matter.
Henceforth,
we parametrize the Stiefel manifold as
\begin{align*}
    \stiefel_d(k) \coloneqq \{\MX \in \R^{d \times k} : \MX^*\MX = \MI_k\},
\end{align*}
where $\MI_k$ denotes the identity in $\R^{k \times k}$,
and where the columns correspond to the basis vectors.
We equip $\R^{d \times k}$ with the Euclidean geometry,
i.e., with
$\smash{\|\MX\|_{\tF} \coloneqq \sum_{j=1}^k \|\Vx_{j}\|^2}$
and
$\smash{\langle \MX, \MY \rangle_{\tF} 
\coloneqq \sum_{j=1}^k \langle \Vx_{j}, \Vy_j \rangle}$
where $\MX = [\Vx_1 | \dots |\Vx_k], 
\MY = [\Vy_1 | \dots |\Vy_k] \in \R^{d \times k}$.
Moreover,
we employ the 1-1-norm
$\|\MX\|_{1,1} \coloneqq \sum_{j=1}^k \|\Vx_{k}\|_1$,
the spectral norm $\lVert \MX \rVert_2$,
and the rank $\rk(\MX)$.
In order to denoise Stiefel-valued data,
we use similar techniques as introduced in \cite{BeBrSt25,BeBr24}.
More precisely,
we propose a TV model for cartoon-like signals
and a Tikhonov model for smooth signals.
For the special case $\stiefel_d(1) \simeq \sphere_{d-1}$,
the proposed denoisers coincide with those in \cite{BeBrSt25,BeBr24}.

\subsection{TV Denoiser for Stiefel-Valued Data}
\label{sec:tv_stiefel_valued_problem}

To restore piecewise constant
Stiefel-valued data $\MX \coloneqq (\MX_n)_{n \in V} \in (\stiefel_d(k))^N$ 
from noisy measurements 
$\MY \coloneqq (\MY_n)_{n \in V} \in (\R^{d \times k})^N$,
we consider the non-convex TV model:
\begin{equation}
    \label{eq:tv_stiefel_non_convex}
    \argmin_{\MX \in \stiefel_d^N(k)}
    \; 
    \tfrac{1}{2} \sum_{n \in V} \|\MX_n - \MY_n\|^2_{\tF}
    + \lambda \TV(\MX),
    \quad 
    \text{where}
    \quad 
    \TV(\MX) 
    \coloneqq 
    \sum_{(n,m) \in E}
    \|\MX_n - \MX_m\|_{1,1},
\end{equation}
which relies on an anisotropic TV regularization 
on the ambient space $\R^{d \times k}$.
In analogy to §~\ref{sec:tv_binary_valued_problem},
using 
$\frac{1}{2} \lVert \MX_n - \MY_n\rVert_{\tF}^2
= - \langle \MX_n, \MY_n \rangle_{\tF} + \text{constant}$, 
the objective may be replaced by
\begin{equation*}
    \FK(\MX)
    \coloneqq 
    - \sum_{n \in V} \langle \MX_n, \MY_n\rangle_{\tF}
    + \lambda \TV(\MX).
\end{equation*}
For confexifying \eqref{eq:tv_stiefel_non_convex},
we encode the non-convex domain
by a series of positive semi-definite, fixed-rank matrices,
where the rank constraint is relaxed afterwards.
This procedure is inspired by similar considerations 
for hyperbolic-valued data \cite{BeBr25}.

\begin{lemma}
    \label{lem:vertice_stiefel}
    Let $n \in V$ and $\MX_n \in \R^{d \times k}$. 
    Then
    $\MX_n \in \stiefel_d(k)$
    if and only if 
    $\MV_n \coloneqq \left[\begin{smallmatrix}
            \MI_d & \MX_n \\
            \MX_n^* & \MI_k
        \end{smallmatrix}\right] \succeq \Mzero$
    and $\rk(\MV_n) = d$.
\end{lemma}
\begin{proof}
    If $\rk(\MV_n) = d$,
    then the last $k$ rows of $\MV_n$
    can be written as linear combinations
    of the independent first $d$ rows.
    Based on this observation,
    we obtain the (positive semi-definite) low-rank representation
    \begin{equation*}
        \MV_n 
        = \left[\begin{smallmatrix}
            \MI_d \\
            \MX_n^*
        \end{smallmatrix}\right]
        \left[\begin{smallmatrix}
            \MI_d & \MX_n
        \end{smallmatrix}\right]
        = 
        \left[\begin{smallmatrix}
            \MI_d & \MX_n \\
            \MX_n^* & \MX_n^*\MX_n
        \end{smallmatrix}\right]
        \quad \text{such that} \quad 
        \MX_n^*\MX_n = \MI_k
        \quad \text{and} \quad 
        \MX_n \in \stiefel_d(k).
    \end{equation*}
    The other way round,
    if $\MX_n \in \stiefel_d(k)$,
    this representation implies $\rk(\MV_n) = d$ and $\MV \succeq \Mzero$.
\end{proof}

\begin{lemma}
    \label{lem:relax-stiefel}
    $\MV_n \succeq \Mzero$
    if and only if $\lVert \MX_n \rVert_2 \le 1$.
\end{lemma}
\begin{proof}
Applying Schur's theorem \cite[p.~495]{HJ13} 
to $\MV_n$ with respect to the upper left identity $\MI_d$ yields 
\begin{equation*}
    \MV_n \succeq \Mzero
    \quad \Leftrightarrow \quad 
    \MV_n / \MI_d = \MI_k - \MX_n^*\MX_n \succeq \Mzero
    \quad \Leftrightarrow \quad 
    \MX_n^*\MX_n \preceq \MI_k
    \quad \Leftrightarrow \quad 
    \|\MX_n\|_2 \leq 1.
    \tag*{\qedhere}
\end{equation*}
\end{proof}

The non-convex TV model \eqref{eq:tv_stiefel_non_convex}
may now be convexified by encoding the domain in $(\MV_n)_{n \in V}$
via Lemma~\ref{lem:vertice_stiefel},
relaxing the constraints on $\MV_n$ by neglecting the rank,
and applying Lemma~\ref{lem:relax-stiefel}.
This procedure yields the \textbf{convexified TV model}:
\begin{equation}
    \label{eq:tv_stiefel_convex_rewitten}
    \argmin_{\MX \in (\R^{d\times k})^N}
    \quad
    \FK(\MX)
    \quad 
    \text{s.t.}
    \quad 
    \|\MX_n\|_2 \leq 1
    \quad 
    \forall n \in V.
\end{equation}

\paragraph{Solving the Convexified TV-Model}

Basically,
the convex minimization problem \eqref{eq:tv_stiefel_convex_rewitten}
can again be solved using ADMM \cite{bauschke2017},
cf. §~\ref{sec:tv_binary_valued_problem}.
More precisely,
relying on the splitting 
\smash{$\FK(\MX) + \iota_{\sphere_{d,k}^N}(\MU)$ with $\MX - \MU = \Mzero$},
where $\sphere_{d,k}$ denotes the unit ball in $\R^{d \times k}$
with respect to the spectral norm, 
we obtain Algorithm~\ref{alg:2}.

\begin{algorithm}{ADMM to solve \eqref{eq:tv_stiefel_convex_rewitten}.}
    \label{alg:2}
    Choose $\MX^{(0)} = \MU^{(0)} = \MZ^{(0)} = \Mzero \in (\R^{d\times k})^N$, 
    step size $\rho > 0$,
    and regularization parameter $\lambda > 0$
    \\
    \textbf{For} $i \in \N$ \textbf{do:}\\
    \hspace*{0.5cm} 
    $\MX^{(i+1)}
        = \prox_{\FK, \frac{1}{\rho}}(\MU^{(i)} - \MZ^{(i)})
        = \prox_{\TV, \frac{\lambda}{\rho}}(\MU^{(i)} - \MZ^{(i)} + \frac{1}{\rho} \, \MY)$ \\
    \hspace*{0.5cm} 
    $\MU^{(i+1)} 
    = \prox_{\iota_{\sphere_{d,k}^N}, \frac{1}{\rho}}(\MX_n^{(i+1)} + \MZ_n^{(i)})
    = \proj_{\sphere_{d,k}^N}(\MX_n^{(i+1)} + \MZ_n^{(i)})$ \\
    \hspace*{0.5cm} 
    $\MZ^{(i+1)} 
    = \MZ^{(i)} + \MX^{(i+1)} - \MU^{(i+1)}$
\end{algorithm}

\begin{remark}
    It is well known that
    the projection onto $\sphere_{d,k}$ can be computed
    using an eigenvalue decomposition
    and projecting the eigenvalues therein onto $[-1,1]$,
    similar to \cite[Thm.~13]{BeBrSt25}.
    Similarly to §~\ref{sec:tv_binary_valued_problem},
    the TV proximal mapping can be efficiently computed by
    applying the fast TV program \cite{Con13v4} coordinatewise.
    The convergence of Algorithm~\ref{alg:2} is again guaranteed 
    by \cite[Cor.~28.3]{bauschke2017}.
\end{remark}

\subsection{Tikhonov Denoiser for Stiefel-Valued Data}
\label{sec:tik_stiefel_valued_problem}

To restore smooth Stiefel-valued data
$\MX \coloneqq (\MX_n)_{n \in V} \in \stiefel_d^N(k)$
from noisy measurements 
$\MY \coloneqq ( \MY_n)_{n \in V} \in (\R^{d \times k})^N$,
we instead consider the non-convex Tikhonov model:
\begin{equation}
    \label{eq:tik_stiefel_non_convex}
    \argmin_{\MX \in \stiefel_d^N(k)}
    \quad
    \tfrac{1}{2} \sum_{n \in V} \|\MX_n - \MY_n\|^2_{\tF}
    + \tfrac{\lambda}{2} 
    \sum_{n \in V} \|\MX_n - \MX_m\|^2_{\tF}.
\end{equation}
Exploiting that,
for $\MX_n, \MX_m \in \stiefel_d(k)$
and fixed $\MY_n \in \R^{d\times k}$,
$\lVert \MX_n - \MY_n \rVert_{\tF} 
=
- \langle \MX_n, \MY_n \rangle_{\tF} 
+ \text{constant}$
and
$\lVert \MX_n - \MX_m \rVert_{\tF} 
=
- \langle \MX_n, \MX_m \rangle_{\tF} 
+ \text{constant}$,
and introducing the auxiliary variables
$\ML_{(n,m)} \coloneqq \MX_n^* \MX_m$,
we rewrite the objective of \eqref{eq:tik_stiefel_non_convex}
using
\begin{equation*}
    \FL(\MX, \ML)
    \coloneqq 
    - \smashoperator{\sum_{n \in V}} \langle \MX_n, \MY_n \rangle_{\tF}
    - \smashoperator{\sum_{(n,m) \in E}} \langle \ML_{(n,m)}, \Mone_{k}\rangle_{\tF},
\end{equation*}
where $\Mone_{k} \in \R^{k\times k}$ denotes the all-one matrix.
In analogy to \cite{BeBrSt25} for sphere-valued data,
we encode the non-convex domain $\MX_n \in \stiefel_d(k)$
and the non-convex side condition $\ML_{(n,m)} \coloneqq \MX_n^* \MX_m$
using positive semi-definite, fixed-rank matrices.

\begin{lemma}
    \label{lem:stiefel_matrix}
    Let $\MX_n, \MX_m \in \R^{d \times k}$ 
    and
    $\ML_{(n,m)} \in \R^{k \times k}$.
    Then $\MX_n, \MX_m \in \stiefel_d(k)$
    and $\ML_{(n,m)} \coloneqq \MX_n^* \MX_m$
    if and only if
    \begin{align*}
        \MQ_{(n,m)} \coloneqq \left[\begin{smallmatrix}
            \MI_d & \MX_n & \MX_m \\
            \MX_n^* & \MI_k & \ML_{(n,m)} \\
            \MX_m^* & \ML_{(n,m)}^* & \MI_k
        \end{smallmatrix}\right]
        \succeq \Mzero
        \quad \text{and} \quad
        \rk(\MQ_{(n,m)}) = d.
    \end{align*}
\end{lemma}

The lemma can be proven 
like Lemma~\ref{lem:vertice_stiefel}
by considering the decomposition
$$\MQ_{(n,m)} = [\MI_d, \MX_n, \MX_m]^* [\MI_d, \MX_n, \MX_m].$$
Removing the rank constraint regarding $\MQ_{(n,m)}$,
we propose to solve the \textbf{convexified Tikhonov model}:
\begin{equation}
    \label{eq:tik_stiefel_convex}
    \argmin_{\substack{\MX \in (\R^{d \times k})^N\\ \ML \in (\R^{k \times k})^M}}
    \quad  
    \FL(\MX, \ML)
    \quad
    \text{s.t.}
    \quad 
    \MQ_{(n,m)} \succeq \Mzero
    \quad 
    \forall
    (n,m) \in E.
\end{equation}

\paragraph{Solving the Convexified Tikhonov Model}

To handle the positive semi-definiteness constraints
on $\MQ_{(n,m)}$ jointly,
we introduce the linear operator 
$\FQ \coloneqq (\FQ_{(n,m)})_{(n,m) \in E}$
with 
\begin{equation*}
    \FQ_{(n,m)} : (\R^{d\times k})^N \times (\R^{k\times k})^M \to \R^{(d+2k) \times (d+2k)},
    \quad 
    (\MX, \ML) \mapsto \MQ_{(n,m)} - \MI_{d+2k}
\end{equation*}
and denote the restriction of its adjoint to the argument $\MX$
by $\FQ_{\MX}^*$
as well as
to $\ML$
by $\FQ_{\ML}^*$.
To solve \eqref{eq:tik_stiefel_convex},
we apply the ADMM \cite{bauschke2017} to the splitting
\smash{$\FL(\MX, \ML) + \iota_{\PSD^M_{d+2k}}(\MU)$}
with $\MU = \FQ(\MX, \ML)$
and $\PSD_{d+2k} \coloneqq
\{\MA \in \R^{(d+2k)\times (d+2k)} 
: \MA \succeq - \MI_{d+2k}\}$.
Relying on similar arguments as in the proof of \cite[Thm.~5.1]{BeBrSt25},
a brief direct calculation yields Algorithm~\ref{alg:3},
where the number of connected vertices is defined as
$\nu_n \coloneqq |\{m \in V : (n,m) \in E \vee (m,n) \in E\}|$.

\begin{algorithm}{ADMM to solve \eqref{eq:tik_stiefel_convex}.}
    \label{alg:3}
    Choose $\MX^{(0)} = \Mzero \in (\R^{d\times k})^N$ and $\MU^{(0)} = \MZ^{(0)} = \Mzero \in (\R^{k\times k})^M$, 
    step size $\rho > 0$
    and regularization parameter $\lambda > 0$
    \\
    \textbf{For} $i \in \N$ \textbf{do:}\\
    \hspace*{0.5cm} 
    $\MX_n^{(i+1)}
        = 
        \tfrac{1}{2\nu_n}\bigl[ \FQ^*_{\MX}(\MU^{(i)} - \MZ^{(i)})_n + \frac{1}{\rho} \, \MY_n  \bigr]
        \quad \forall n \in V$
    \\
    \hspace*{0.5cm} 
    $\ML_{(n,m)}^{(i+1)}
        = 
        \tfrac{1}{2} \bigl[\FQ_{\ML}^*(\MU^{(i)} - \MZ^{(i)})_{(n,m)} + \frac{\lambda}{\rho} \, \ME \bigr]
        \quad\forall (n,m) \in E$
    \\
    \hspace*{0.5cm} 
    $\MU^{(i+1)}_{(n,m)} 
    = 
    \proj_{\PSD_{d + 2k}} \bigl(\FQ_{(n,m)}(\MX^{(i+1)}, \ML^{(i+1)}) 
    + \MZ_{(n,m)}^{(i)} \bigr)
    \quad \forall (n,m) \in E$ \\
    \hspace*{0.5cm} 
    $\MZ^{(i+1)} 
    = \MZ^{(i)} + \FQ(\MX^{(i+1)}, \ML^{(i+1)}) - \MU^{(i+1)}$
\end{algorithm}

\begin{remark}
    The projection onto $\PSD_{d+2k}$ can be computed
    using an eigenvalue decomposition
    and projecting the eigenvalues therein onto $[-1,\infty)$.
    The convergence of Algorithm~\ref{alg:3} is again guaranteed 
    by \cite[Cor.~28.3]{bauschke2017}.
\end{remark}

\section{Numerical Results}
\label{sec:num}

All algorithms are implemented\footnote{
The code is available at GitHub \url{https://github.com/JJEWBresch/relaxed_tikhonov_regularization}.} 
in Python~3.11.9 using Numpy~1.25.0 and Scipy~1.11.1.
The following experiments are performed on an off-the-shelf MacBook Pro 2020
with Intel Core i5 Chip (4‑Core CPU, 1.4~GHz) and 8~GB RAM.
In every experiment,
we use a range of regularization parameters 
and finally choose the parameter 
yielding the smallest mean squared error (MSE)
in sense of the restoration.

\paragraph{Multi-Binary-Valued Data}

A real-world application 
for multi-binary data denoising
is the restoration 
of multi-color QR codes
\cite{Ta18}.
Figuratively,
these QR codes are generated
by combining three independent QR codes
in the RBG (red-green-blue) color space.
The modules of the resulting QR codes
are then colored in black, red, green, blue,
cyan, magenta, yellow, and white%
---%
corresponding to the vertices of the RGB color space.
To fit our setting,
we scale and shift the RGB color space
such that
the module colors can be interpreted as $\Binary_3$ data.
For a proof-of-concept,
we employ a $20\times20$, multi-color QR code
that is distorted by 
independent, 3-dimensional Gaussian noise on a 10-times finer grid,
yielding an $\R^3$-valued image.
The ground truth 
and the noisy multi-color QR code are reported in Figure~\ref{fig:binary}
(first, second). 
Applying Algorithm~\ref{alg:1},
we obtain the restored multi-color QR code in
Figure~\ref{fig:binary} (third).
Interestingly,
Algorithm~\ref{alg:1} already yields
a $\Binary_3$-valued restoration;
therefore,
the projection procedure 
outlined in Theorem~\ref{thm:binary_prob_tightness_tv}
is here not required.
Although we here have to deal with severe noise,
the reconstruction is nearly perfect
up to some single pixels.

\begin{figure}
    \includegraphics[width=\linewidth,clip=true, trim=150pt 45pt 150pt 55pt]{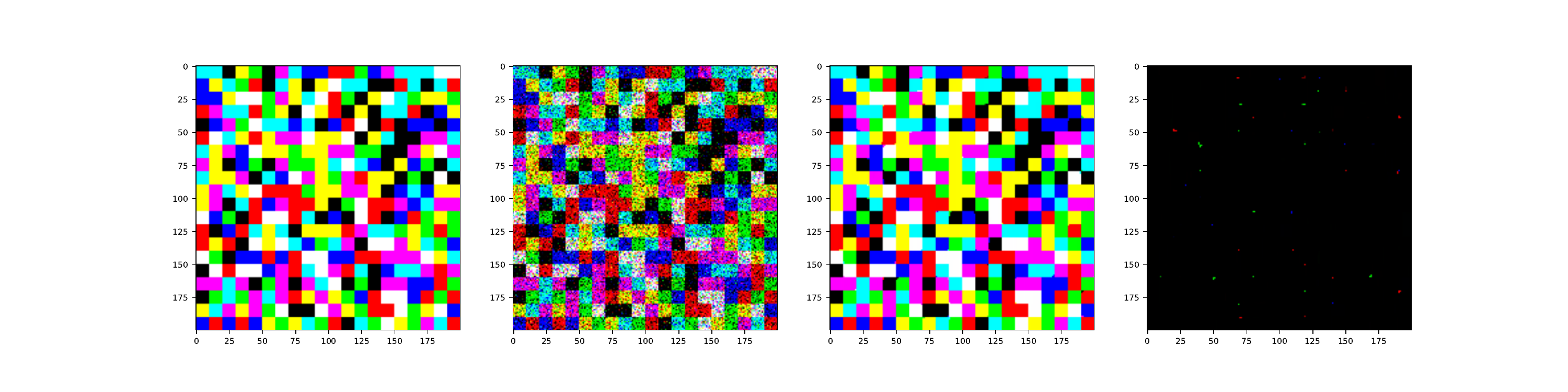}    
    \caption{
    Multi-color QR code denoising
    (from left to right): 
    ground truth (first),
    corrupted image by (3-dimensional) Gaussian noise
    with standard deviation $\sqrt{2}\cdot0.5$ (second),
    restoration using Algorithm~\ref{alg:1}
    with $\lambda = 1.2, \rho=0.1$ (third),
    and the restoration error (forth).
    The mean absolute distance of
    the restored pixels to $\Binary_3$
    is in order of $10^{-5}$,
    i.e., here
    Algorithm~\ref{alg:1} actually yields
    a solution of the non-convex TV model
    \eqref{eq:tv_binary_non_convex}
    without the projection procedure 
    outlined in Theorem~\ref{thm:binary_prob_tightness_tv}.}
    \label{fig:binary}
\end{figure}

\paragraph{Stiefel-Valued Data}
For a proof of concept,
we study the TV model \eqref{eq:tv_stiefel_convex_rewitten}
and the Tikhonov model \eqref{eq:tik_stiefel_convex}
for a synthetic $\stiefel_3(2)$-valued signal of length 200.
Our ground truth is reported in Figure~\ref{fig:stiefel} (left).
The vectors in the orthogonal bases are then distorted 
using the von Mises--Fischer distribution
and are reorthonormalized 
employing the Gram--Schmidt process.
The resulting $\stiefel_3(2)$-valued noisy signal
is reported in Figure~\ref{fig:stiefel} (middle left).
To remove the noise,
we apply our proposed convexified TV model (Algorithm~\ref{alg:2}) 
and Tikhonov model (Algorithm~\ref{alg:3}).
The restored 
$\stiefel_3(2)$-signals 
are reported in Fig.~\ref{fig:stiefel} (middle right, right).
For both models,
we observe convergence to the Stiefel manifold 
in the sense that the norm of basis vectors in $\MX_n^{(i)}$
converges to one 
and that the inner product of the basis vectors 
converges to zero.
Consequently,
both algorithms here actually yield 
minimizers of the non-convex
TV model \eqref{eq:tv_stiefel_non_convex}
and Tikhonov model \eqref{eq:tik_stiefel_non_convex},
respectively.
Especially,
the Tikhonov model nearly recovers the ground truth.

\begin{figure}
\resizebox{\textwidth}{!}{
    \begin{tabular}{c c c c}
        \includegraphics[width=0.48\linewidth, clip=true, trim=140pt 70pt 700pt 70pt]{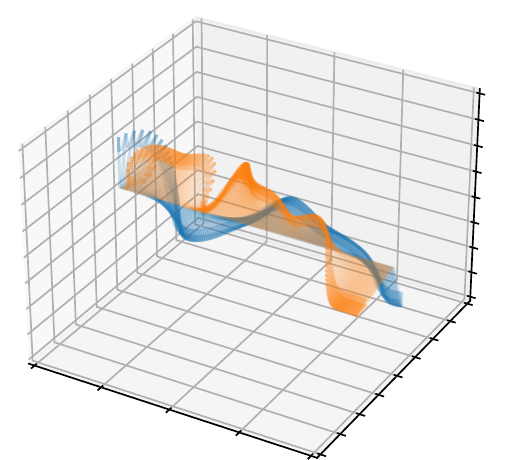}
        \includegraphics[width=0.48\linewidth, clip=true, trim=440pt 70pt 400pt 70pt]{images/grassmannian/grassmannian_TV_signal.pdf}
        \includegraphics[width=0.48\linewidth, clip=true, trim=740pt 70pt 100pt 70pt]{images/grassmannian/grassmannian_TV_signal.pdf}
        \includegraphics[width=0.48\linewidth, clip=true, trim=740pt 70pt 100pt 70pt]{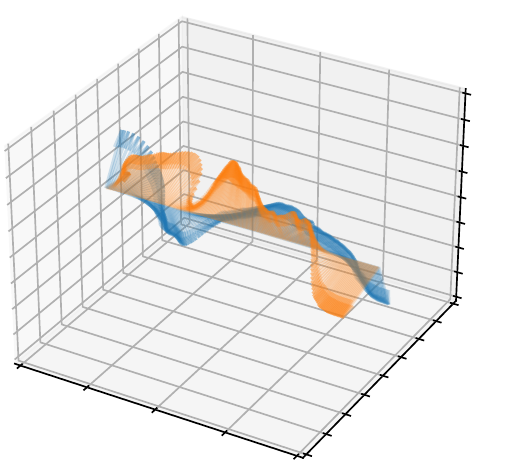}
    \end{tabular}}
    \caption{
    Restoration of a $\stiefel_3(2)$-valued signal (from left to tight):
    ground truth (left),
    noisy measurement (middle left),
    restored signal using Algorithm~\ref{alg:2}
    (convexified TV model)
    with $\lambda=0.75, \rho=0.5$ (middle right),
    and using Algorithm~\ref{alg:3}
    (convexified Tikhonov model)
    with $\lambda=10, \rho=0.1$ (right).
    In both cases 
    the mean absolute distance of the norm 
    of the recovered basis vectors to one is of order $10^{-5}$,
    the mean absolute distance of the inner product
    between the recovered basis vectors to zero 
    is of order $10^{-4}$.}
    \label{fig:stiefel}
\end{figure}

\bibliographystyle{abbrv}
\bibliography{literatur_arXiv}

\end{document}